\documentclass[leqno,12pt]{amsart}
\usepackage{amssymb}
\usepackage{amsmath}
\usepackage{enumerate}
\usepackage{amsfonts}
\usepackage{hyperref}

\usepackage[headheight=18pt, top=25mm, bottom=25mm, left=25mm, right=25mm]{geometry}

\usepackage{tikz}

\definecolor{darkgreen}{RGB}{45, 119, 75}

\newcommand{\supp}{\text{supp }}

\newcommand{\BMO}{{\rm BMO}}
\newcommand{\VMO}{{\rm VMO}}

\newtheorem{theorem}{Theorem}[section]

\newtheorem{lemma}[theorem]{Lemma}

\numberwithin{equation}{section}

\hypersetup{
    colorlinks=true,
    linkcolor= blue,
    citecolor =cyan,
    urlcolor = teal,
}

\begin{document}

\title[Commutators]{A note on commutators of singular integrals with $\BMO$ and $\VMO$ functions in the Dunkl setting}


\author[Jacek Dziubański]{Jacek Dziubański}
\author[Agnieszka Hejna]{Agnieszka Hejna}
\begin{abstract}  On $
\mathbb R^N$ equipped with a root system $R$, multiplicity function $k \geq 0$, and the associated measure $dw(\mathbf{x})=\prod_{\alpha \in R}|\langle \mathbf{x},\alpha\rangle|^{k(\alpha)}\,d\mathbf{x}$, we consider a (non-radial) kernel ${K}(\mathbf{x})$ which has properties similar to those from the classical theory of singular integrals and the Dunkl convolution operator $\mathbf{T}f=f*K$ associated with ${K}$. Assuming that $b$ belongs to the  ${\rm BMO}$ space on the space of homogeneous type  $X=(\mathbb{R}^N,\|\cdot\|,dw)$, we prove that the commutator $[b,\mathbf{T}]f(\mathbf{x})=b(\mathbf{x})\mathbf{T}f(\mathbf{x})-\mathbf{T}(bf)(\mathbf{x})$ is a bounded operator on $L^p(dw)$ for all $1<p<\infty$. Moreover, $[b,\mathbf T]$ is compact on $L^p(dw)$, provided $b\in {\rm VMO} (X)$. The paper extents results of Han, Lee, Li and Wick. 
\end{abstract}

\address{Jacek Dziubański, Uniwersytet Wroc\l awski,
Instytut Matematyczny,
Pl. Grunwaldzki 2,
50-384 Wroc\l aw,
Poland}
\email{jdziuban@math.uni.wroc.pl}

\address{Agnieszka Hejna, Uniwersytet Wroc\l awski,
Instytut Matematyczny,
Pl. Grunwaldzki 2,
50-384 Wroc\l aw,
Poland 
\&
Department of Mathematics,
Rutgers University,
Piscataway, NJ 08854-8019, USA}
\email{hejna@math.uni.wroc.pl}

\maketitle

\section{Introduction }\label{Sec:Intro}
 \subsection{Introduction} Consider $\mathbb R^N$ equipped with a root system $R$ and a non-negative multiplicity function $k\geq 0$. 
 Let $dw(\mathbf x)=\prod_{\alpha\in R} |\langle \alpha, \mathbf x\rangle |^{k(\alpha)}\;d\mathbf x$
  be the associated measure. For $f\in L^1_{loc}(dw)$ and a measurable bounded set $E\subset \mathbb R^N$, we denote
  \begin{equation}\label{eq:mean}
      f_E:=\frac{1}{w(E)}\int_E f(\mathbf x)\, dw(\mathbf x).
  \end{equation}
  Let $G$ be the Coxeter group generated  by the reflections $\sigma_\alpha$, 
  $\alpha\in R$. For  $E\subset \mathbb R^N$, we set 
  $$\mathcal O(E)=\{\sigma(\mathbf x): \sigma\in G, \ \mathbf x\in E\}. $$
 In {Han, Lee, Li, and Wick }~\cite{HLLW} the authors investigated  two types of ${\rm BMO}$ and ${\rm VMO}$ spaces in the Dunkl setting which are connected  with two distances: the Euclidean distance $\|\mathbf x-\mathbf y\|$ and the orbit distance $d(\mathbf x,\mathbf y)=\min_{\sigma\in G} \| \mathbf x-\sigma(\mathbf y)\|$. The spaces ${\rm BMO}$ and ${\rm BMO}_d$ are  defined as 
 \begin{align*}
      {\rm BMO}=\{b\in L^1_{\rm loc} (dw): \|b\|_{{\rm BMO}}<\infty\}, \ \ \| b\|_{{\rm BMO}}:=\sup_{B}\frac{1}{w(B)}\int_B |b(\mathbf x)- b_B|\, dw(\mathbf{x}),
 \end{align*}
\begin{align*}
  {\rm BMO}_d=\{b\in L^1_{\rm loc}(dw): \| b\|_d<\infty\}, \ \   \| b\|_d:=\sup_{B} \frac{1}{w(\mathcal O(B))}\int_{\mathcal O(B)}|b(\mathbf x)-b_{{\mathcal O(B)}}|\, dw(\mathbf x),
\end{align*}
where the supremum  is taken over all Euclidean balls $B=B(\mathbf y,r)=\{\mathbf z\in\mathbb R^N: \|\mathbf y-\mathbf z\|<r\}$. The space ${\rm BMO}_d$ is a proper subspace of ${\rm BMO}$ (see~\cite{JL}) and 
 $$\| b\|_{\BMO}\leq C\|b\|_{\BMO_d}\quad \text{for } b\in \BMO_d.$$ 
 In~\cite{HLLW}  commutators of $\BMO$ and $\BMO_d$ functions with the Dunkl-Riesz transforms $R_j$ are studied. The Dunkl-Riesz transforms are Calder\'on-Zygmund type operators which are formally 
 defined by $R_j=T_{e_j}(-\Delta_k)^{-1/2}$, where $T_{e_j}$ are the Dunkl 
 operators (see~\eqref{eq:T_xi}) and $\Delta_k=\sum_{j=1}^N T_{e_j}^2$ is the Dunkl Laplacian. They were studied by Thangavelu and Xu~\cite{ThangaveluXu1} (in dimension 1 and in the product case) and by Amri and Sifi~\cite{AS} (in higher dimensions)  who proved their bounds on $L^p(dw)$ spaces. One of the main results of~\cite{HLLW} asserts that if $b\in \BMO_d$, then the commutator $[b,R_j]f(\mathbf x)=b(\mathbf x)R_jf(\mathbf x)-R_j(b(\cdot)f(\cdot))(\mathbf x)$ is a bounded operator on $L^p(dw)$ for $1<p<\infty$ and 
 \begin{equation}\label{eq:BMO_d_upper}
 \| [b,R_j]\|_{L^p(dw)\to L^p(dw)}\lesssim \| b\|_{\BMO_d}.
 \end{equation}
Conversely, if for $b\in L^1_{\rm loc}(dw)$, the commutator $[b,R_j]$ is bounded on $L^p(dw)$ for some $1<p<\infty$, then $b\in \BMO$ and 
\begin{equation}\label{eq:lower_Han}
\|b\|_{\BMO}\lesssim \| [b,R_j]\|_{L^p(dw)\to L^p(dw)}.
\end{equation} 
The authors of~\cite{HLLW} raised the  question if the possible lower bound $  \|b\|_{{\rm BMO}_d}\lesssim \| [b,R_j]\|_{L^p(dw)\to L^p(dw)}$ holds true. 

Our first goal in this note is to improve \eqref{eq:BMO_d_upper} by showing that it holds for $b\in \BMO$, that is, there is a constant $C_p>0$
 such that 
\begin{equation}\label{eq:BMO_upper}
 \| [b,R_j]\|_{L^p(dw)\to L^p(dw)}\leq C_p\| b\|_{{\rm BMO}}
 \end{equation}
 (see Theorem \ref{teo:bounded}). 
 
 Let us point out that \eqref{eq:BMO_upper} gives a negative answer to the question formulated above, because otherwise we would get  $\|b\|_{\BMO_d}\lesssim  \| [b,R_j]\|_{L^p(dw)\to L^p(dw)}\lesssim \| b\|_{\BMO}$, which is impossible (see~\cite{JL}).

An essential part of~\cite{HLLW} is devoted for studying compactness of the commutators of ${\rm VMO}$ functions with the Dunkl-Riesz transforms. The ${\rm VMO}$ and ${\rm VMO}_d$ spaces are defined as the closures of the sets of compactly supported  functions from the Lipschitz spaces $\Lambda$ and $\Lambda_d$ in the norms $\|\cdot\|_{{\rm BMO}}$ and $\|\cdot\|_{{\rm BMO}_d}$ respectively. Then ${\rm VMO}_d\subset {\rm VMO}$ and $\|b\|_{{\rm VMO}}\lesssim \|b\|_{{\rm VMO}_d}$ and, thanks to~\cite[Theorem 4.1]{CW}, the dual space to ${\rm VMO}$ is the Hardy space $H^1$ considered in~\cite{ADzH} and~\cite{DH-atom}.  Theorem 1.5 of~\cite{HLLW} states that if $b\in {\rm VMO}_d$,  then 
the commutator $[b,R_j]$ is a compact operator on $L^p(dw)$ for all $1<p<\infty$. Our second aim is to extend this result for all $b\in {\rm VMO}$ (see Theorem \ref{teo:compact}). 
Actually we will prove \eqref{eq:BMO_d_upper} and the compactness result for commutators $[b,\mathbf T]$  of $\BMO$ and $\VMO$ functions with  Dunkl singular integral operators $\mathbf T$  of convolution type (under certain regularity for the associated kernels $ K(\mathbf x)$). The Dunkl--Riesz transforms are the basic examples of such operators. 

Let us remark that for the Riesz transforms, the lower bounds~\eqref{eq:lower_Han} proved in~\cite{HLLW} together with the upper bounds~\eqref{eq:BMO_upper} (see Theorem~\ref{teo:bounded}) generalize (to the Dunkl setting) the classical results of Coifman, Rochberg and Weiss~~\cite{CRW} and Janson~\cite{Janson} obtained  on the Euclidean spaces $(\mathbb R^N, \, d\mathbf x)$. As far as the compactness is concerned, the necessity result~\cite[Theorem 1.5]{HLLW} together with the sufficiency result (see Theorem~\ref{teo:compact}) extend to the Dunkl theory the classical theorems of Uchiyama~\cite{Uchiyama} about the characterization of the ${\rm VMO}$ functions  by commutators with the Riesz transforms.  

\section{Preliminaries}

\subsection{Dunkl theory}

In this section we present basic facts concerning the theory of the Dunkl operators.   For more details we refer the reader to~\cite{Dunkl},~\cite{Roesle99},~\cite{Roesler3}, and~\cite{Roesler-Voit}. 

We consider the Euclidean space $\mathbb R^N$ with the scalar product $\langle \mathbf{x},\mathbf y\rangle=\sum_{j=1}^N x_jy_j
$, where $\mathbf x=(x_1,...,x_N)$, $\mathbf y=(y_1,...,y_N)$, and the norm $\| \mathbf x\|^2=\langle \mathbf x,\mathbf x\rangle$.

A {\it normalized root system}  in $\mathbb R^N$ is a finite set  $R\subset \mathbb R^N\setminus\{0\}$ such that $R \cap \alpha \mathbb{R} = \{\pm \alpha\}$,  $\sigma_\alpha (R)=R$, and $\|\alpha\|=\sqrt{2}$ for all $\alpha\in R$, where $\sigma_\alpha$ are defined by 
\begin{equation}\label{reflection}\sigma_\alpha (\mathbf x)=\mathbf x-2\frac{\langle \mathbf x,\alpha\rangle}{\|\alpha\|^2} \alpha.
\end{equation}

The finite group $G$ generated by the reflections $\sigma_{\alpha}$, $\alpha \in R$, is called the {\it Coxeter group} ({\it reflection group}) of the root system.

A~{\textit{multiplicity function}} is a $G$-invariant function $k:R\to\mathbb C$ which will be fixed and $\geq 0$  throughout this paper.  

The associated measure $dw$ is defined by $dw(\mathbf x)=w(\mathbf x)\, d\mathbf x$, where 
 \begin{equation}\label{eq:measure}
w(\mathbf x)=\prod_{\alpha\in R}|\langle \mathbf x,\alpha\rangle|^{k(\alpha)}.
\end{equation}
Let $\mathbf{N}=N+\sum_{\alpha \in R}k(\alpha)$. Then, 
\begin{equation}\label{eq:t_ball} w(B(t\mathbf x, tr))=t^{\mathbf{N}}w(B(\mathbf x,r)) \ \ \text{\rm for all } \mathbf x\in\mathbb R^N, \ t,r>0.
\end{equation}
Observe that there is a constant $C>0$ such that for all $\mathbf{x} \in \mathbb{R}^N$ and $r>0$ we have
\begin{equation}\label{eq:balls_asymp}
C^{-1}w(B(\mathbf x,r))\leq  r^{N}\prod_{\alpha \in R} (|\langle \mathbf x,\alpha\rangle |+r)^{k(\alpha)}\leq C w(B(\mathbf x,r)),
\end{equation}
so $dw(\mathbf x)$ is doubling. Moreover, there exists a constant $C\ge1$ such that,
for every $\mathbf{x}\in\mathbb{R}^N$ and for all $r_2\ge r_1>0$,
\begin{equation}\label{eq:growth}
C^{-1}\Big(\frac{r_2}{r_1}\Big)^{N}\leq\frac{{w}(B(\mathbf{x},r_2))}{{w}(B(\mathbf{x},r_1))}\leq C \Big(\frac{r_2}{r_1}\Big)^{\mathbf{N}}.
\end{equation}

For $\xi \in \mathbb{R}^N$, the {\it Dunkl operators} $T_\xi$  are the following $k$-deformations of the directional derivatives $\partial_\xi$ by   difference operators:
\begin{equation}\label{eq:T_xi}
     T_\xi f(\mathbf x)= \partial_\xi f(\mathbf x) + \sum_{\alpha\in R} \frac{k(\alpha)}{2}\langle\alpha ,\xi\rangle\frac{f(\mathbf x)-f(\sigma_\alpha(\mathbf{x}))}{\langle \alpha,\mathbf x\rangle}.
\end{equation}
The Dunkl operators $T_{\xi}$, which were introduced in~\cite{Dunkl}, commute and are skew-symmetric with respect to the $G$-invariant measure $dw$. 

The closures of connected components of
\begin{equation*}
    \{\mathbf{x} \in \mathbb{R}^{N}\;:\; \langle \mathbf{x},\alpha\rangle \neq 0 \text{ for all }\alpha \in R\}
\end{equation*}
are called (closed) \textit{Weyl chambers}. We remark that $\|\mathbf x-\mathbf y\|=d(\mathbf x,\mathbf y)$ if and only if $\mathbf x,\mathbf y \in \mathbb{R}^N$ belong to the same closed Weyl chamber. 

Let \begin{equation}
    Mf(\mathbf{x})=\sup_{B \ni \mathbf{x}}\frac{1}{w(B)}\int_{B}|f(\mathbf{y})|\,dw(\mathbf{y}),
\end{equation}
denote the (uncentered) Hardy--Littlewood maximal function on $(\mathbb{R}^N,\|\cdot\|,dw)$.

\subsection{Dunkl transform}
For fixed $\mathbf y\in\mathbb R^N$, the {\it Dunkl kernel} $\mathbf{x} \longmapsto E(\mathbf x,\mathbf y)$ is a unique solution to the system 
$$T_\xi f=\langle \xi ,\mathbf y\rangle f, \quad f(0)=1.$$
The function $E(\mathbf x,\mathbf y)$, which generalizes the exponential function $e^{\langle\mathbf x,\mathbf y\rangle}$, has a unique extension to a holomorphic  function $E(\mathbf z,\mathbf w)$ on $\mathbb C^N\times \mathbb C^N$. 
Let $f \in L^1(dw)$. We define the \textit{Dunkl transform }$\mathcal{F}f$ of $f$ by
\begin{equation}\label{eq:Dunkl_transform}
    \mathcal{F} f(\xi)=\mathbf{c}_k^{-1}\int_{\mathbb{R}^N}f(\mathbf{x})E(\mathbf{x},-i\xi)\, {dw}(\mathbf{x}), \ \  \mathbf{c}_k=\int_{\mathbb{R}^N}e^{-\frac{{\|}\mathbf{x}{\|}^2}2}\,{dw}(\mathbf{x}){>0}
\end{equation}
The Dunkl transform is a generalization of the Fourier transform on $\mathbb{R}^N$. It was introduced in~\cite{D5} for $k \geq 0$ and further studied in~\cite{dJ1} in the more general context.  It was proved in~\cite[Corollary 2.7]{D5} (see also~\cite[Theorem 4.26]{dJ1}) that it extends uniquely  to  an isometry on $L^2(dw)$.
We have also the following inversion theorem (\cite[Theorem 4.20]{dJ1}): for all $f \in L^1(dw)$ such that $\mathcal{F}f \in L^1(dw)$ we have $f(\mathbf{x})=(\mathcal{F})^2f(-\mathbf{x}) \text{ for almost all }\mathbf{x} \in \mathbb{R}^N\textup{.}$ The inverse $\mathcal F^{-1}$ of $\mathcal{F}$ has the form
  \begin{equation}\label{eq:inverse_teo} \mathcal F^{-1} f(\mathbf{x})=\mathbf{c}_k^{-1}\int_{\mathbb R^N} f(\xi)E(i\xi, \mathbf x)\, dw(\xi)= \mathcal{F}f(-\mathbf{x})\quad \text{\rm for } f\in L^1(dw)\textup{.}
  \end{equation}

\subsection{Dunkl translations}
Suppose that $f \in \mathcal{S}(\mathbb{R}^N)$ {
(the Schwartz class of functions on $\mathbb R^N$)} and $\mathbf{x} \in \mathbb{R}^N$. We define the \textit{Dunkl translation }$\tau_{\mathbf{x}}f$ of $f$ to be
\begin{equation}\label{eq:translation}
    \tau_{\mathbf{x}} f(-\mathbf{y})=\mathbf{c}_k^{-1} \int_{\mathbb{R}^N}{E}(i\xi,\mathbf{x})\,{E}(-i\xi,\mathbf{y})\,\mathcal{F}f(\xi)\,{dw}(\xi)=\mathcal{F}^{-1}(E(i\cdot,\mathbf{x})\mathcal{F}f)(-\mathbf{y}).
\end{equation}
The Dunkl translation was introduced in~\cite{R1998}. The definition can be extended to functions which are not necessary in $\mathcal{S}(\mathbb{R}^N)$. For instance, thanks to the Plancherel's theorem, one can define the Dunkl translation of $L^2(dw)$ function $f$ by
\begin{equation}\label{eq:translation_Fourier}
    \tau_{\mathbf{x}}f(-\mathbf{y})=\mathcal{F}^{-1}(E(i\cdot,\mathbf{x})\mathcal{F}f(\cdot))(-\mathbf{y})
\end{equation}
(see~\cite{R1998} of~\cite[Definition 3.1]{ThangaveluXu}). In particular, the operators $f \mapsto \tau_{\mathbf{x}}f$ are contractions on $L^2(dw)$. Here and subsequently, for a reasonable function $g(\mathbf{x})$, we  write $g(\mathbf x,\mathbf y):=\tau_{\mathbf x}g(-\mathbf y)$.

We will need the following result concerning the support of the Dunkl translation of a compactly supported function.

\begin{theorem}[{\cite[Theorem 1.7]{DzH}}]\label{teo:support}
 Let $f \in L^2(dw)$, $\text{\rm supp}\, f \subseteq B(0,r)$, and $\mathbf{x} \in \mathbb{R}^N$. Then  
\begin{equation}\label{eq:inclusion_support}
\text{\rm supp}\, \tau_{\mathbf{x}}f(-\, \cdot) \subseteq \mathcal{O}(B(\mathbf x,r)).
\end{equation}
\end{theorem}

\subsection{Dunkl convolution}
Assume that $f,g \in L^2(dw)$. The \textit{generalized convolution} (or \textit{Dunkl convolution}) $f*g$ is defined by the formula
\begin{equation}\label{eq:conv_transform}
    f*g(\mathbf{x})=\mathbf{c}_k\mathcal{F}^{-1}\big((\mathcal{F}f)(\mathcal{F}g)\big)(\mathbf{x}),
\end{equation}
equivalently, by
\begin{equation}\label{eq:conv_translation}
    (f*g)(\mathbf{x})=\int_{\mathbb{R}^N}f(\mathbf{y})\,\tau_{\mathbf{x}}g(-\mathbf{y})\,{dw}(\mathbf{y})=\int_{\mathbb{R}^N}g(\mathbf{y})\,\tau_{\mathbf{x}}f(-\mathbf{y})\,{dw}(\mathbf{y}).
\end{equation}
Generalized convolution of $f,g \in \mathcal{S}(\mathbb{R}^N)$ was considered in~\cite{R1998} and~\cite{Trimeche}, the definition was extended to $f,g \in L^2(dw)$ in~\cite{ThangaveluXu}.

\subsection{Singular integral kernels}

Let us consider a (non-radial) kernel $K(\mathbf{x})$ which has properties similar to those
from the classical theory. Namely, let $s_0$ be an even positive integer larger than $\mathbf{N}$, which will be fixed in the whole paper. Consider a function  $K\in C^{s_0} (\mathbb R^N\setminus \{0\})$ such that 
\begin{equation}\label{eq:uni_on_annulus}\tag{A} \sup_{0<a<b<\infty} \Big| \int_{a<\|\mathbf x\|<b} K(\mathbf x)\, dw(\mathbf x)\Big|<\infty,  
\end{equation} 
\begin{equation}\label{eq:assumption1}\tag{D} 
\Big|\frac{\partial^\beta}{\partial \mathbf x^\beta} K(\mathbf x)\Big|\leq C_{\beta}\|\mathbf x\|^{-\mathbf N-|\beta|} \quad \text{for all} \ |\beta |\leq s_0, 
\end{equation}
\begin{equation}\label{eq:limitA}\tag{L} 
     \lim_{\varepsilon \to 0} \int_{\varepsilon <\|\mathbf x\|<1} K(\mathbf x)\, dw(\mathbf x)=L \text{  for some  }L \in \mathbb{C}.
     \end{equation} 
Set 
\begin{align*}
    K^{\{t\}}(\mathbf x)=K(\mathbf x)(1-\phi(t^{-1} \mathbf x)),    
\end{align*}
where $\phi$ is a  fixed  radial $C^\infty$-function supported by the unit ball $B(0,1)$ such that $\phi (\mathbf x)=1$ for $\|\mathbf x\|<1/2$. The following theorem was proved in~\cite{singular}.

\begin{theorem}\label{toe:from_singular}
Assume that~\eqref{eq:uni_on_annulus} and \eqref{eq:assumption1} are satisfied, then
\begin{enumerate}[(i)]
    \item{\textup{(}Theorems 4.1 and 4.2 of~\cite{singular}\textup{)}  the operators $f \mapsto f*K^{\{t\}}$ are bounded on $L^p(dw)$ for $1<p<\infty$ and they are of weak--type $(1,1)$ with the bounds independent of $t>0$;} \label{numitem:bounded_truncated}
    \item{\textup{(}Theorems 3.7 and 4.3 of~\cite{singular}\textup{)} assuming additionally~\eqref{eq:limitA}, the  limit $\lim_{t\to 0} f*K^{\{t\}} (\mathbf x)$ exists and defines a bounded operator $\mathbf T$ on $L^p(dw)$ for $1<p<\infty$, which is of weak-type (1,1) as well.}\label{numitem:bounded}
\end{enumerate}
\end{theorem}

\section{Statement of the results}

\subsection{Commutators}

In order to define the commutator operator, we come back to the definition of the limit operator $\mathbf{T}$. Let $0<\varepsilon<\min(1,s_0-\mathbf{N})$. For any $t>0$ let us denote $K^{\{t/2,t\}}:=K^{\{t/2\}}-K^{\{t\}}$.
Then $K^{\{t/2,t\}}$ is $C^{s_0}(\mathbb R^N)$-function supported by $B(0,t) \setminus B(0,t/4)$ (cf.~\cite[(3.1)]{singular}). In order to simplify the notation, we write $K_\ell=K^{\{2^{\ell-1},2^{\ell}\}}$. It is proved in~\cite[(4.24), (4.25)]{DzH_nonradial} that for all $\mathbf{x},\mathbf{y} \in \mathbb{R}^N$ and $\ell \in \mathbb{Z}$, we have

\begin{equation}\label{eq:scaled_CZ1}
    |K_{\ell}(\mathbf{x},\mathbf{y})| \leq {{C}}\left(1+\frac{\|\mathbf{x}-\mathbf{y}\|}{2^{\ell}}\right)^{-\varepsilon}w(B(\mathbf{x},2^{\ell}))^{-1/2}w(B(\mathbf{y},2^{\ell}))^{-1/2},
\end{equation}
\begin{equation}\label{eq:scaled_CZ2}
\begin{split}
    &|K_{\ell}(\mathbf{x},\mathbf{y})-K_{\ell}(\mathbf{x},\mathbf{y}')| \\&\leq {{C}}\frac{\|\mathbf{y}-\mathbf{y}'\|^{\varepsilon}}{2^{\varepsilon \ell}}\left(1+\frac{\|\mathbf{x}-\mathbf{y}\|}{2^{\ell}}\right)^{-\varepsilon}w(B(\mathbf{x},2^{\ell}))^{-1/2}\left(w(B(\mathbf{y},2^{\ell}))^{-1/2}+w(B(\mathbf{y}',2^{\ell}))^{-1/2}\right).
\end{split}
\end{equation}
Moreover, $K_{\ell}(\mathbf{x},\mathbf{y})=K_{\ell}( -\mathbf{y}, -\mathbf{x})$  and by~\cite[proof of Theorem 4.6]{DzH_nonradial},
\begin{equation}\label{eq:sum}
    \sum_{\ell=-\infty}^{\infty}|K_{\ell}(\mathbf{x},\mathbf{y})| \leq Cw(B(\mathbf x,d(\mathbf x,\mathbf y)))^{-1}\frac{d(\mathbf x,\mathbf y)^{\varepsilon}}{\| \mathbf x-\mathbf y\|^{\varepsilon}}
\end{equation}
for all $\mathbf{x},\mathbf{y} \in \mathbb{R}^N$ such that $d(\mathbf{x},\mathbf{y}) \neq 0$ and 
\begin{equation}\label{eq:sum_lip}
    \sum_{\ell=-\infty}^{\infty}|K_{\ell}(\mathbf{x},\mathbf{y})-K_{\ell}(\mathbf{x},\mathbf{y}')| \leq C\frac{\|\mathbf{y}-\mathbf{y}'\|^{\varepsilon}}{\| \mathbf x-\mathbf y\|^{\varepsilon}}w(B(\mathbf x,d(\mathbf x,\mathbf y)))^{-1}
\end{equation}
for all $\mathbf{x},\mathbf{y},\mathbf{y}' \in \mathbb{R}^N$ such that $\|\mathbf{y}-\mathbf{y}'\|<\frac{d(\mathbf{x},\mathbf{y})}{2}$. The number $\varepsilon>0$ will be fixed in the whole paper. 
Thanks to \eqref{eq:sum} the function 
\begin{equation}
    \label{eq:kernel_K}
    K(\mathbf x,\mathbf y)=\sum_{\ell=-\infty}^\infty K_\ell(\mathbf x,\mathbf y)
\end{equation}
is well defined for $d(\mathbf x,\mathbf y)>0$ and, by Theorem \ref{toe:from_singular}, it is the associated kernel to the operator $\mathbf T$, that is, 
\begin{equation}\label{eq:T-kenrel} \mathbf Tf(\mathbf x)=\int_{\mathbb R^N} K(\mathbf x,\mathbf y)f(\mathbf y)\, dw(\mathbf y)
\end{equation}
for $f\in L^p(dw)$ and $\mathbf x\notin \text{supp}\, f$. 

{Let us emphasize that the estimate for $K(\mathbf x,\mathbf y)$ which are consequences of~\eqref{eq:sum} and~\eqref{eq:sum_lip} turn out to be very useful in handling some harmonic analysis problems in the Dunkl setting (see~\cite{Tan}). }

From now on we fix a kernel  $K\in C^{s_0} (\mathbb R^N\setminus \{0\})$ satisfying~\eqref{eq:uni_on_annulus},~\eqref{eq:limitA}, and ~\eqref{eq:assumption1} for some $s_0>\mathbf N$.  Let  $b \in {\rm BMO}$.  For any  compactly supported $f \in L^{p}(dw)$ for some $p>1$, we  define
\begin{equation}\label{eq:comm}
\mathcal{C}f(\mathbf{x})={[b,\mathbf T]} f(\mathbf{x})=\lim_{m \to \infty}\int_{\mathbb{R}^N}(b(\mathbf{x})-b(\mathbf{y}))\sum_{\ell=-m}^{m}K_{\ell}(\mathbf{x},\mathbf{y})f(\mathbf{y})\,dw(\mathbf{y})=\lim_{m \to \infty}\mathcal{C}_mf(\mathbf{x}) 
\end{equation}
 The existence of the limit in~\eqref{eq:comm} in any $L^{p_0}(dw)$-norm, provided $1< p_0<p$, is proved in Lemma~\ref{lem:in_Lp0}. Then, Theorem~\ref{teo:bounded} and its proof  allows one to extend the definition for all $f \in L^p(dw)$. 
\subsection{Statements of main theorems}
Our main results are the following two theorems.

\begin{theorem}\label{teo:bounded}
Let $p>1$. Assume that a kernel  $K\in C^{s_0} (\mathbb R^N\setminus \{0\})$ satisfies~\eqref{eq:uni_on_annulus},~\eqref{eq:limitA},~\eqref{eq:assumption1} for a certain even integer $s_0>\mathbf N$, and $b \in {\rm BMO}$. Then there is a constant $C>0$ independent of $b$ such that for all $f \in L^p(dw)$ we have
\begin{equation}\label{eq:comm_main}
    \|\mathcal{C}f\|_{L^p(dw)} \leq C\|b\|_{{\rm BMO}}\|f\|_{L^p(dw)}.
\end{equation}
\end{theorem}

In order to formulate our second theorem, recall that ${\rm VMO}$ is the closure in ${\rm BMO}$ of compactly supported Lipschitz functions, i.e., functions $f$ satisfying $ \sup_{\mathbf{x} \neq \mathbf{y}}\frac{|f(\mathbf{x})-f(\mathbf{y})|}{\|\mathbf{x}-\mathbf{y}\|}<\infty$.

\begin{theorem}\label{teo:compact} Assume $b \in {\rm VMO}$. For $1<p<\infty$, the commutator $\mathcal{C}$ is a compact operator on $L^p(dw)$.
\end{theorem}

 To prove the theorems, we adapt the ideas of classical proofs (cf. e.g.~\cite{Grafakos}) to apply estimates for integral kernels of operators which are expressed in terms of the orbit distance  $d(\mathbf x,\mathbf y)$ and the Euclidean metric, see \eqref{eq:sum} and \eqref{eq:sum_lip} (cf. also~\cite{HLLW}). These  require in some places much  careful analysis. For example, we use the Coxeter group for decomposing  $L^p(dw)$-functions (see \eqref{eq:decomp_f}) or we split integration over $\mathbb R^N$ onto the Weyl chambers (see~\eqref{eq:spitGamma}). 
 
 \section{Proof of Theorem\texorpdfstring{~\ref{teo:bounded}}{3.1}}
 
Let us begin with the following lemma. Recall that $b_{B(\mathbf{x},r)}$ is defined by~\eqref{eq:mean}.

\begin{lemma}\label{lem:log}
There is a constant $C>0$  such that for all $b\in L^1_{{\rm loc}}(dw)$, $r_1>r>0$, $\mathbf{x},\mathbf{y} \in \mathbb{R}^N$, and $\sigma \in G$, we have:
\begin{equation}\label{eq:log_non}
    |b_{B(\mathbf{x},r)}-b_{B(\mathbf{x},r_1)}| \leq C\log(r_1/r)\|b\|_{{\rm BMO}},
\end{equation}

\begin{equation}\label{eq:not-far}
    |b_{B(\mathbf{x},r)}-b_{B(\mathbf{y},r)}| \leq C\|b\|_{{\rm BMO}} \quad \text{\rm provided } \|\mathbf x-\mathbf y\|\leq 2r,
\end{equation}

\begin{equation}\label{eq:log_sigma}
    |b_{B(\mathbf{x},r)}-b_{B(\sigma(\mathbf{x}),r)}| \leq C\log\left(\frac{\|\sigma(\mathbf{x})-\mathbf{x}\|}{r}+4\right)\|b\|_{{\rm BMO}},
\end{equation}
(John--Nirenberg inequality) for any $1\leq s<\infty$ there is $C_s>0$ such that for all positive integers $j$ we have
\begin{equation}\label{eq:John}
   \Big(\frac{1}{w(B(\mathbf x,2^jr))} \int_{B(\mathbf x,2^jr)}|b(\mathbf y)-b_{B(\mathbf x,r)}|^s\, dw(\mathbf y) \Big)^{1/s}\leq C_s j \|b\|_{\rm BMO}.
\end{equation}
\end{lemma}

\begin{proof}
Inequalities~\eqref{eq:log_non}, \eqref{eq:not-far}, and~\eqref{eq:John} are consequences of the well-known results in metric spaces with doubling measures (see e.g.~\cite{Buckley},~\cite{CW}, or~\cite[Section 4]{JL} in the Dunkl setting). In order to prove~\eqref{eq:log_sigma}, note that if $\|\mathbf{x}-\sigma(\mathbf{x})\| \leq 2r$, then~\eqref{eq:log_sigma} follows by~\eqref{eq:not-far}. Assume that $\|\mathbf{x}-\sigma(\mathbf{x})\|>2r$ and let $j$ be the smallest positive integer such that $\|\mathbf{x}-\sigma(\mathbf{x})\| \leq 2^j r$. Then $B(\mathbf{x},2^{j+2}r) \cap  B(\sigma(\mathbf{x}),2^{j+2}r)\ne \emptyset$.
So, applying~\eqref{eq:log_non} and~\eqref{eq:not-far}, we get 
\begin{equation*}
    \begin{split}
        &|b_{B(\mathbf x,r)}-b_{B(\sigma(\mathbf x), r)}|\\
        &\leq  |b_{B(\mathbf x,r)}-b_{B(\mathbf x, 2^{j+2}r)}|+  |b_{B(\mathbf x,2^{j+2}r)}-b_{B(\sigma(\mathbf x), 2^{j+2}r)}|+ |b_{B(\sigma(\mathbf x),2^{j+2}r)}-b_{B(\sigma(\mathbf x), r)}|\\
       & \leq Cj\|b\|_{\rm BMO}\leq C' \log\left(\frac{\|\mathbf{x}-\sigma(\mathbf{x})\|}{r}+4\right)\|b\|_{{\rm BMO}}.
    \end{split}
\end{equation*}

\end{proof}

 \begin{lemma}\label{lem:in_Lp0}
 Let $p>p_0>1$. If $f \in L^{p}(dw)$ is compactly supported, then $\mathcal{C}_m f \in L^{p_0}(dw)$ for all $m \in \mathbb{Z}$ and
 \begin{align*}
     \lim_{m \to \infty}\mathcal{C}_mf=\mathcal{C}f \text{ in }L^{p_0}(dw).
 \end{align*}
 \end{lemma}
 
 \begin{proof}
 Let  $B=B(0,r)$ be any ball such that ${\rm supp \;}f \subseteq B$. We write
 \begin{align*}
     \mathcal{C}_m f(\mathbf{x})=(b(\mathbf{x})-b_B)\sum_{\ell=-m}^{m}K_{\ell}*f(\mathbf{x})+\sum_{\ell=-m}^{m}K_{\ell}*((b-b_B)f))(\mathbf{x}).
 \end{align*}
 By H\"older's inequality with $s=p/p_0$ and the John-Nirenberg inequality~\eqref{eq:John} we obtain  
 \begin{align*}
     &\left(\int_{B}|b(\mathbf{x})-b_B|^{p_0}|f(\mathbf{x})|^{p_0}\;dw(\mathbf{x})\right)^{1/p_0} \leq Cw(B)^{1/(p_0s')}\|b\|_{\BMO}\|f\|_{L^p(dw)},
 \end{align*}
 where $s'>1$ is such that $s+s'=ss'$, so $(b-b_B)f \in L^{p_0}(dw)$. Hence, by Theorem~\ref{toe:from_singular}~\eqref{numitem:bounded}, 
  $$ \lim_{m \to \infty}\sum_{\ell=-m}^{m}K_{\ell}*((b-b_B)f))$$  exists in the $L^{p_0}(dw)$-norm and it is equal to $\mathbf{T}((b-b_B)f) \in L^{p_0}(dw)$. Again, by  Theorem~\ref{toe:from_singular}~\eqref{numitem:bounded}, the limit 
 $$ \lim_{m \to \infty}\sum_{\ell=-m}^{m}K_{\ell}*f$$ 
 exists in the $L^{p_0}(dw)$ and $L^p(dw)$-norms and it is equal to $\mathbf{T}f$. We write
 \begin{align*}
  (b(\mathbf{x})-b_{B})\mathbf{T}f(\mathbf{x})= (b(\mathbf{x})-b_{B})\mathbf{T}f(\mathbf{x})\chi_{5B}(\mathbf{x})+ (b(\mathbf{x})-b_{B})\mathbf{T}f(\mathbf{x})\chi_{(5B)^c}(\mathbf{x})=:g_{1}(\mathbf{x})+g_{2}(\mathbf{x}).   
 \end{align*}
 We will show that $g_1,g_2 \in L^{p_0}(dw)$. For $g_1$, by H\"older's inequality with $s=p/p_0$, the John-Nirenberg inequality~\eqref{eq:John}, and $L^p(dw)$-boundedness of $\mathbf{T}$ (Theorem~\ref{toe:from_singular}~\eqref{numitem:bounded}), we get
 \begin{align*}
     &\left(\int_{5B}|b(\mathbf{x})-b_B|^{p_0}|\mathbf{T}f(\mathbf{x})|^{p_0}\;dw(\mathbf{x})\right)^{1/p_0}\\& \leq C\left(\int_{5B}|b(\mathbf{x})-b_B|^{p_0s'}\,dw(\mathbf{x})\right)^{1/(p_0s')}\left(\int_{5B}|\mathbf Tf(\mathbf x)|^{p}\;dw(\mathbf{x})\right)^{1/p}\\&\leq Cw(B)^{1/(p_0s')}\|b\|_{\BMO}\|f\|_{L^p(dw)}.
 \end{align*}
In order to prove that $g_2 \in L^{p_0}(dw)$, we observe that, by~\eqref{eq:t_ball} and~\eqref{eq:balls_asymp}, $w(B(\mathbf x,d(\mathbf x,\mathbf y)))\sim w(B(0,\|\mathbf x\|))=c\|\mathbf x\|^{\mathbf N}$ for $\mathbf{y} \in B$ and $\mathbf{x} \in (5B)^c$. Using~\eqref{eq:sum} we have 
 \begin{equation}
 \begin{split}
   \chi_{ (5B)^c}(\mathbf x)| \mathbf Tf(\mathbf x)|
   & \leq C\chi_{ (5B)^c}(\mathbf x)\int_{B}\frac{|f(\mathbf x)|}{w(B(\mathbf x, d(\mathbf x,\mathbf y)))}\, dw(\mathbf y)\\
   &\leq C\int_{B}\frac{|f(\mathbf x)|}{w(B(0, \|\mathbf x\|)))}\, dw(\mathbf y)\leq C\|f\|_{L^p(dw)}\frac{w(B)^{1/p'}}{w(B(0,\|\mathbf x\|))}.\\
 \end{split}
 \end{equation}
Consequently, 
\begin{equation}
    \begin{split}
        \|g_2\|_{L^{p_0}(dw)}^{p_0}
        &\leq C \| f\|_{L^p(dw)}^{p_0}w(B)^{p_0/p'}\int_{\|\mathbf x\|>5r}\frac{|b(\mathbf x)-b_B|^{p_0}}{w(B(0,\|\mathbf x\|))^{p_0}}\, dw(\mathbf x)\\
        &\leq C \| f\|_{L^p(dw)}^{p_0}w(B)^{p_0/p'}\sum_{j=1}^\infty \int_{\|\mathbf x\|\leq 2^jr}\frac{|b(\mathbf x)-b_B|^{p_0}}{w(B(0,2^{j}r))^{p_0}}\, dw(\mathbf x).
    \end{split}
\end{equation}
 Now, applying the John-Nirenberg inequality~\eqref{eq:John} we get 
 \begin{equation}
    \begin{split}
        \|g_2\|_{L^{p_0}(dw)}^{p_0}
        &\leq C \| f\|_{L^p(dw)}^{p_0}w(B)^{p_0/p'}\sum_{j=1}^\infty \| b\|_{\BMO}^{p_0}\frac{j^{p_0}}{(2^jr)^{(p_0-1)\mathbf N}}\\
        &\leq C' \| f\|_{L^p(dw)}^{p_0}w(B)^{p_0/p'}\| b\|_{\BMO}^{p_0} r^{-(p_0-1)\mathbf N}. \\       
    \end{split}
\end{equation}
 \end{proof}

\begin{proof}[Proof of Theorem\texorpdfstring{~\ref{teo:bounded}}{ 3.1}]
We shall prove the inequality~\eqref{eq:comm_main} for compactly supported functions $f\in L^p(dw)$ which form a dense subspace in $L^p(dw)$. To this end, thanks to Lemma~\ref{lem:in_Lp0} and~\cite[page 148, Theorem 2 and the remark below it]{St2}, it suffices to verify that 
\begin{equation}\label{eq:sharp}
     \|(\mathcal{C}f)^{\#}\|_{L^p(dw)}\leq C\| b\|_{\BMO} \| f\|_{L^p(dw)},
\end{equation}
where $g^{\#}$ denotes the sharp maximal function 
$$ g^{\#}(\mathbf x)=\sup_{B\ni \mathbf x} \inf_{c\in \mathbb C} \frac{1}{w(B)} \int_{B} |g(\mathbf y)-c|\, dw(\mathbf y).$$
  
Let $\mathbf{x}\in \mathbb R^N$ and let $B=B(\mathbf{x}_0,r)$ be any ball which contains $\mathbf{x}$. We write all the elements of $G\setminus\{{\rm id}\}$ in a sequence  $\sigma_1,\sigma_2,\ldots,\sigma_{|G|-1}$. We define the sets $U_j \subseteq \mathbb{R}^N$, $j=1,2,\ldots,|G|-1$, inductively:
\begin{equation*}
\begin{split}
   & U_1:=\{\mathbf{z}\in \mathbb{R}^N\;:\; \| \mathbf{z}-\mathbf{x}_0\|>5r, \ \|\mathbf{z}-\sigma_1(\mathbf{x}_0)\|\leq 5r\},\\
    &U_{j+1}:=\{\mathbf{z}\in \mathbb{R}^N\;:\; \| \mathbf{z}-\mathbf{x}_0\|>5r,\ \|\mathbf{z}-\sigma_{j+1}(\mathbf{x}_0)\|\leq 5r\}\setminus \left(\bigcup_{j_1=1}^jU_{j_1}\right) \text{ for }j \geq 1.
\end{split}
\end{equation*}  
For a compactly supported  function $f\in L^p(dw)$,  we decompose 
\begin{equation}\label{eq:decomp_f}
    f=f_1+f_2+\sum_{j=1}^{|G|-1} f_{\sigma_j}, \text{ where } f_1:=f\cdot \chi_{5B}, \quad f_2:=f\cdot\chi_{\mathcal (\mathcal{O}(5B))^c}, \quad f_{\sigma_j} :=f\cdot\chi_{U_j}.
\end{equation}
For $\mathbf{y}\in B$ we set
 \begin{equation*}
 \begin{split}
     &g_1(\mathbf{y}):=\mathcal{C}f_1(\mathbf{y})=(b(\mathbf{y})-b_B)\mathbf{T}f_1(\mathbf{y})+\mathbf{T}((b_B-b)f_1)(\mathbf{y})=:g_{11}(\mathbf{y})+g_{12}(\mathbf{y}),\\
      &g_2(\mathbf{y}):=\mathcal{C}f_2(\mathbf{y})=(b(\mathbf{y})-b_B)\mathbf{T}f_2(\mathbf{y})+\mathbf{T}((b_B-b)f_2)(\mathbf{y})=:g_{21}(\mathbf{y})+g_{22}(\mathbf{y}),\\
      &g_{\sigma_j}(\mathbf{y}):=\mathcal{C}f_{\sigma_j}(\mathbf{y})=(b(\mathbf{y})-b_B)\mathbf{T}f_{\sigma_{j}}(\mathbf{y})+\mathbf{T}((b_B-b)f_{\sigma_j})(\mathbf{y})=:g_{\sigma_j1}(\mathbf{y})+g_{\sigma_j2}(\mathbf{y}).\\
     \end{split}
 \end{equation*}
 Fix $1<s<p$. Further, by the fact that $|(g_{11})_B|\leq \frac{1}{w(B)}\int_B |g_{11}(\mathbf{y})|\,dw(\mathbf{y})$, by the definition of $g_{11}$, and H\"older's inequality,
 \begin{equation*}
     \begin{split}
         \frac{1}{w(B)} & \int_B |g_{11}(\mathbf{y})-(g_{11})_B|\;dw(\mathbf{y})\leq \frac{2}{w(B)}\int_B |g_{11}(\mathbf{y})|\, dw(\mathbf{y})\\
         &\leq 2\frac{1}{w(B)}\int_B |b(\mathbf{y})-b_B| \cdot |\mathbf{T}f_1(\mathbf{y})|\, dw(\mathbf{y})\\
         & \leq 2 \Big(\frac{1}{w(B)}\int_B |b(\mathbf{y})-b_B|^{s'}dw(\mathbf{y})\Big)^{1/s'} \cdot \Big(\frac{1}{w(B)}\int_B|\mathbf{T}f_1(\mathbf{y})|^s\, dw(\mathbf{y})\Big)^{1/s},
     \end{split}
 \end{equation*}
where $s'>1$ is such that $\frac{1}{s}+\frac{1}{s'}=1$. Applying inequality~\eqref{eq:John} we conclude that 
 \begin{equation*}
     g_{11}^{\#} (\mathbf{x})\leq C\|b\|_{\BMO} (M(|\mathbf{T}f_1|^s)(\mathbf{x}))^{1/s}.
 \end{equation*}
 The same analysis gives
 \begin{equation*}
 \begin{split}
     &g_{21}^{\#} (\mathbf{x})\leq C\|b\|_{\BMO} (M(|\mathbf{T}f_2|^s)(\mathbf{x}))^{1/s}, \ \ \ g_{\sigma_j1}^{\#} (\mathbf{x})\leq C\|b\|_{\BMO} (M(|\mathbf{T}f_{\sigma_j}|^s)(\mathbf{x}))^{1/s}.\\
     \end{split}
 \end{equation*}
 To deal with $g_{12}$, we choose $q,v\in (1,\infty)$ such that $1<qv<p<\infty$ and  $s =qv$. Then, by H\"older's inequality and $L^{q}(dw)$-boundedness of $\mathbf{T}$,
 \begin{equation*}
\begin{split}
         \frac{1}{w(B)} & \int_B |g_{12}(\mathbf{y})-(g_{12})_B|\;dw(\mathbf{y})\leq \frac{2}{w(B)}\int_B |g_{12}(\mathbf{y})|\, dw(\mathbf{y})\\
         &\leq 2\frac{1}{w(B)}\int_B |\mathbf{T}((b(\cdot)-b_B) \cdot f_1)(\mathbf{y})|\, dw(\mathbf{y})\\
         &\leq 2 \Big(\frac{1}{w(B)}\int_B \mathbf{T}((b(\cdot)-b_B) \cdot f_1)(\mathbf{y})|^q\, dw(\mathbf{y})\Big)^{1/q} \\
         &\leq C \Big(\frac{1}{w(5B)}\int_{5B} |(b(\mathbf{y})-b_B) \cdot f_1(\mathbf{y})|^q\, dw(\mathbf{y})\Big)^{1/q} \\
         & \leq C\Big(\frac{1}{w(5B)}\int_{5B} |b(\mathbf{y})-b_B|^{qv'}dw(\mathbf{y})\Big)^{1/(qv')} \cdot \Big(\frac{1}{w(5B)}\int_{5B}|f_1(\mathbf{y})|^{qv}\, dw(\mathbf{y})\Big)^{1/(qv)}.
     \end{split}
 \end{equation*}
 Hence, applying the John-Nirenberg inequality~\eqref{eq:John}, we get 
  \begin{equation*}
     g_{12}^{\#} (\mathbf{x})\leq C\|b\|_{\BMO} (M(|f_1|^s)(\mathbf{x}))^{1/s }.
 \end{equation*}
We turn to analyse $g_{22}$. Observe that  for $\mathbf{z}\notin \mathcal O(5B)$ and $\mathbf{y} \in B$ we have  $\| \mathbf{x}_0-\mathbf{y}\|\leq d(\mathbf{x}_0,\mathbf{z})$. 
 Let $\Gamma$ be a fixed closed  Weyl chamber such that $\mathbf{x}_0\in \Gamma$, then by the estimates~\eqref{eq:sum_lip},
 \begin{equation}\label{eq:spitGamma}
     \begin{split}
         |g_{22}(\mathbf{y})-g_{22}(\mathbf{x}_0)|&\leq \int_{\mathbb{R}^N}  \sum_{\ell=-\infty}^{\infty}|K_{\ell}(\mathbf{z},\mathbf{y})-K_{\ell}(\mathbf{z},\mathbf{x}_0)||b_B-b(\mathbf{z})||f_2(\mathbf{z})|\, dw(\mathbf{z})\\
         &=\sum_{\sigma\in G}\int_{\mathbf{z}\in \sigma (\Gamma)}  \sum_{\ell=-\infty}^{\infty}|K_{\ell}(\mathbf{z},\mathbf{y})-K_{\ell}(\mathbf{z},\mathbf{x}_0)||b_B-b(\mathbf{z})||f_2(\mathbf{z})|\, dw(\mathbf{z})\\
         &\leq C\sum_{\sigma\in G}\int_{\mathbf{z}\in \sigma (\Gamma)} \frac{\|\mathbf{y}-\mathbf{x}_0\|^{\varepsilon}}{\|\mathbf{x}_0-\mathbf{z}\|^{\varepsilon}} \frac{1}{w(B(\mathbf{x}_0,d(\mathbf{x}_0,\mathbf{z}))}|b_B-b(\mathbf{z})||f_2(\mathbf{z})|\, dw(\mathbf{z})\\
         &=:\sum_{\sigma\in G} J_\sigma(\mathbf{y}).\\
     \end{split}
 \end{equation}
In dealing with  $J_\sigma(\mathbf{y})$ we shall use the inequalities:
\begin{equation*}
    \begin{split}
       & \|\mathbf{x}_0-\mathbf{z}\|\geq \max(\|\mathbf{x}_0-\sigma(\mathbf{x}_0)\|/2,r)\quad \text{for } \mathbf z\in \sigma (\Gamma),\\
      & \| \sigma(\mathbf x_0)-\sigma (\mathbf x)\|\leq r<5r \leq \|\sigma (\mathbf x_0)-\mathbf z\| =d(\mathbf x_0,\mathbf z)\leq \|\mathbf x_0-\mathbf z\| \quad \text{for } \mathbf z\in \sigma (\Gamma), \ \mathbf z\notin \mathcal O(5B).
    \end{split}
\end{equation*}
So,  
 \begin{equation*}
     \begin{split}
         &J_\sigma(\mathbf{y})\leq C\int_{\mathbf{z}\in \sigma(\Gamma)} \frac{r^{\varepsilon}}{\| \mathbf{x}_0-\mathbf{z}\|^{\varepsilon}} \frac{1}{w(B(\sigma(\mathbf{x}_0), \| \sigma(\mathbf{x}_0)-\mathbf{z}\|))} |b_B-b_{\sigma(B)}||f_2(\mathbf{z})|\, dw(\mathbf{z})\\
         & \ \ + C\int_{\mathbf{z}\in \sigma(\Gamma)} \frac{r^{\varepsilon}}{\| \mathbf{x}_0-\mathbf{z}\|^{\varepsilon}} \frac{1}{w(B(\sigma(\mathbf{x}_0), \| \sigma(\mathbf{x}_0)-\mathbf{z}\|))} |b_{\sigma(B)}-b(\mathbf{z})||f_2(\mathbf{z})|\, dw(\mathbf{z})=:J_{\sigma,1}(\mathbf{y})+J_{\sigma,2}(\mathbf{y}).
     \end{split}
 \end{equation*}
 Further, by~\eqref{eq:log_sigma},
 \begin{equation}\label{eq:J_sigma_1}
     \begin{split}
         J_{\sigma,1} (\mathbf{y}) &\leq C\int_{\mathbf{z}\in \sigma(\Gamma)} \frac{r^{\varepsilon/2}}{\|\mathbf{x}_0-\mathbf{z}\|^{\varepsilon/2}} \frac{r^{\varepsilon/2}}{r^{\varepsilon/2}+\|\mathbf{x}_0-\sigma(\mathbf{x}_0)\|^{\varepsilon/2}}\log \Big(\frac{\|\mathbf{x}_0-\sigma (\mathbf{x}_0)\|}{r}+4\Big)\| b\|_{\BMO}\\
         &\hskip1cm \times \frac{|f_2(\mathbf{z})|}{w(B(\sigma(\mathbf{x}_0),\|\sigma(\mathbf{x}_0)-\mathbf{z}\|))}\, dw(\mathbf{z})\\
         &\leq C\|b\|_{\BMO} \int_{\mathbf{z}\in \sigma(\Gamma)} \frac{r^{\varepsilon/2}}{\|\mathbf{x}_0-\mathbf{z}\|^{\varepsilon/2} }\cdot \frac{|f_2(\mathbf{z})|}{w(B(\sigma(\mathbf{x}_0),\|\sigma(\mathbf{x}_0)-\mathbf{z}\|))}\, dw(\mathbf{z})\\
         &\leq C\|b\|_{\BMO} \sum_{j=2}^\infty \int_{\mathbf{z}\in \sigma(\Gamma), \|\sigma(\mathbf x_0)-\mathbf z\|\sim 2^jr} \frac{r^{\varepsilon/2}}{\|\mathbf{x}_0-\mathbf{z}\|^{\varepsilon/2} }\cdot \frac{|f_2(\mathbf{z})|}{w(B(\sigma(\mathbf{x}_0),\|\sigma(\mathbf{x}_0)-\mathbf{z}\|))}\, dw(\mathbf{z})\\
         &\leq C\| b\|_{\BMO} Mf_2(\sigma (\mathbf{x})).
     \end{split}
 \end{equation}
We turn to considering $J_{\sigma,2}(\mathbf{y})$. 
 Applying H\"older's inequality and then the John-Nirenberg inequality~\eqref{eq:John} we obtain 
\begin{equation}\label{eq:J_sigma_2}
    \begin{split}
        J_{\sigma,2}(\mathbf{y})&\leq C\sum_{j=2}^\infty \int_{\mathbf{z}\in \sigma(\Gamma),\; \|\mathbf z-\sigma(\mathbf x_0)\|\sim 2^jr} \frac{r^{\varepsilon}}{2^{\varepsilon j}r^{\varepsilon}}|b_{\sigma(B)}-b(\mathbf{z})| \frac{|f_2(\mathbf{z})|}{w(B(\sigma(\mathbf{x}_0),2^jr))}\, dw(\mathbf{z})\\
        &\leq C\sum_{j=2}^\infty 2^{-\varepsilon j} \Big(\int_{\|\mathbf z-\sigma(\mathbf x_0)\|\sim 2^jr} |b_{\sigma(B)}-b(\mathbf z)|^{s'}\frac{dw(\mathbf z)}{w(B(\sigma (\mathbf x_0),2^jr))}\Big)^{1/s'}\\
        &\ \ \times \Big(\int_{\|\mathbf z-\sigma(\mathbf x_0)\|\leq  2^jr} |f_2(\mathbf z)|^{s}\frac{dw(\mathbf z)}{w(B(\sigma (\mathbf x_0),2^jr))}\Big)^{1/s}\\ 
        &\leq C\sum_{j=2}^\infty 2^{-\varepsilon j} j \| b\|_{\BMO} (M(|f_2|^s)(\sigma(\mathbf x)))^{1/s}\leq C\| b\|_{\BMO} (M(|f_2|^s)(\sigma (\mathbf{x})))^{1/s}.
    \end{split}
\end{equation}

Thus, by~\eqref{eq:J_sigma_1} and~\eqref{eq:J_sigma_2} we have got 
$$ g_{22}^{\#}(\mathbf{x})\leq C\sum_{\sigma\in G} \Big(Mf_2(\sigma (\mathbf{x})) + (M(|f_2|^s)(\sigma (\mathbf{x})))^{1/s}\Big).$$
Finally we turn to estimate  $g_{\sigma_j2}$. To this end we note that for $\mathbf z\in U_j$ and $\mathbf y\in B$ we have 
\begin{align*}
    \|\mathbf{z}-\mathbf{y}\| \geq \|\mathbf{z}-\mathbf{x}_0\|-\|\mathbf{x}_0-\mathbf{y}\| \geq 5r-r=4r,
\end{align*}
\begin{align*}
    \|\mathbf{x}_0-\sigma_j(\mathbf{x}_0)\| \leq \|\mathbf{x}_0-\mathbf{y}\|+\|\mathbf{z}-\mathbf{y}\|+\|\mathbf{z}-\sigma_{j}(\mathbf{x}_0)\| \leq 6r+\|\mathbf{z}-\mathbf{y}\| \leq \frac{5}{2}\|\mathbf{z}-\mathbf{y}\|,
\end{align*} 
Consequently, by~\eqref{eq:sum},
\begin{equation}\label{eq:R_on_Uj}
\begin{split}
    &\int_B  \sum_{\ell=-\infty}^{\infty}|K_{\ell}(\mathbf{z},\mathbf{y})|\, dw(\mathbf{y}) \leq C\int_B \frac{d(\mathbf{z},\mathbf{y})^{\varepsilon}}{\|\mathbf{z}-\mathbf{y}\|^{\varepsilon}}\frac{1}{w(B(\mathbf{z},d(\mathbf{z},\mathbf{y})))}\,dw(\mathbf{y})\\&\leq C\frac{r^{\varepsilon}}{(r+\|\mathbf{x}_0-\sigma_j(\mathbf{x}_0)\|)^{\varepsilon}}\int_B \frac{d(\mathbf{z},\mathbf{y})^\varepsilon}{r^{\varepsilon}}\frac{1}{w(B(\mathbf{z},d(\mathbf{z},\mathbf{y})))}\,dw(\mathbf{y})\\&\leq C\frac{r^{\varepsilon}}{(r+\|\mathbf{x}_0-\sigma_j(\mathbf{x}_0)\|)^{\varepsilon}}\int_{\mathcal{O}(B(\mathbf{z},16r))} \frac{d(\mathbf{z},\mathbf{y})^\varepsilon}{r^{\varepsilon}}\frac{1}{w(B(\mathbf{z},d(\mathbf{z},\mathbf{y})))}\,dw(\mathbf{y})\\&\leq C\frac{r^{\varepsilon}}{(r+\|\mathbf{x}_0-\sigma_j(\mathbf{x}_0)\|)^{\varepsilon}}\sum_{j=-4}^{\infty}\int_{d(\mathbf{z},\mathbf{y}) \sim 2^{j}r} \frac{d(\mathbf{z},\mathbf{y})^\varepsilon}{r^{\varepsilon}}\frac{1}{w(B(\mathbf{z},d(\mathbf{z},\mathbf{y})))}\,dw(\mathbf{y})\\&\leq C\frac{r^{\varepsilon}}{(r+\|\mathbf{x}_0-\sigma_j(\mathbf{x}_0)\|)^{\varepsilon}}\sum_{j=-4}^{\infty}{2^{-\varepsilon j}}\int_{d(\mathbf{z},\mathbf{y}) \sim 2^{-j}r} \frac{1}{w(B(\mathbf{z},2^{-j}r))}\,dw(\mathbf{y})\\&\leq C\frac{r^{\varepsilon}}{(r+\|\mathbf{x}_0-\sigma_j(\mathbf{x}_0)\|)^{\varepsilon}}.
\end{split}
\end{equation}
Hence, by~\eqref{eq:R_on_Uj} and~\eqref{eq:log_sigma}, 
\begin{equation*}
\begin{split}
    \frac{1}{w(B)} &\int_B |g_{\sigma_j2}(\mathbf{y})  -(g_{\sigma_j2})_B|\, dw(\mathbf{y})\leq \frac{2}{w(B)} \int_B  |g_{\sigma_j2}(\mathbf{y})|\, dw(\mathbf{y})\\
    &\leq \frac{2}{w(B)}\int_B \int_{U_j} \sum_{\ell=-\infty}^{\infty}|K_{\ell}(\mathbf{y},\mathbf{z})| |b_B-b(\mathbf{z})|\cdot |f_{\sigma_j}(\mathbf{z})|\, dw(\mathbf{z})\, dw(\mathbf{y})\\
    &\leq C \frac{r^{\varepsilon}}{(r+\|\sigma_j(\mathbf{x}_0)-\mathbf{x}_0\|)^{\varepsilon}}\frac{1}{w(B)} \int_{U_j} |b_B-b(\mathbf{z})|\cdot |f_{\sigma_j}(\mathbf{z})|\, dw(\mathbf{z})\\
    &\leq C\frac{r^{\varepsilon}}{(r+\|\sigma_j(\mathbf{x}_0)-\mathbf{x}_0\|)^{\varepsilon}}\Big(\frac{1}{w(B)}\int_{U_j}|b_B-b(\mathbf{z})|^{s'}\, dw(\mathbf{z})\Big)^{1/s'}\Big(\frac{1}{w(B)}\int_{U_j}|f_{\sigma_j}(\mathbf{z})|^s\, dw(\mathbf{z})\Big)^{1/s}\\
    &\leq C\frac{r^{\varepsilon}}{(r+\|\sigma_j(\mathbf{x}_0)-\mathbf{x}_0\|)^{\varepsilon}} \log\Big(\frac{\|\sigma_j({\mathbf{x}_0})-\mathbf{x}_0\|}{r}+4\Big)\| b\|_{\BMO} (M|f_{\sigma_j}|^s(\sigma(\mathbf{x})))^{1/s}, 
    \end{split}
\end{equation*}
so 
$$ |g_{\sigma_j2}^{\#}(\mathbf{x})|\leq C\| b\|_{\BMO} (M|f_{\sigma_j}|^s(\sigma(\mathbf{x})))^{1/s}.$$ 
Finally we end up with the estimate 
\begin{equation}\label{eq:finsl_sharp}
    \begin{split}
(\mathcal{C}f)^{\#}(\mathbf{x})&\leq C \|b\|_{\BMO} \Big( (M(|\mathbf{T}f_1|^s)(\mathbf{x})))^{1/s}+M(|\mathbf{T}f_2|^s(\mathbf{x})))^{1/s}\\ 
&+ \sum_{j=1}^{|G|-1}M(|\mathbf{T}f_{\sigma_j}|^s(\mathbf{x})))^{1/s} +\sum_{\sigma\in G} \big(Mf(\sigma(\mathbf{x})) +  (M|f|^s(\sigma(\mathbf{x})))^{1/s}\big) \Big).
    \end{split}
\end{equation}
Hence, using the  $L^{p_1}(dw)$-boundedness of the  Hardy--Littlewood maximal $M$ function for all $1<p_1<\infty$ and the fact that the measure $dw$ is $G$-invariant, we conclude~\eqref{eq:sharp} from \eqref{eq:finsl_sharp} and \eqref{eq:decomp_f}.
 \end{proof}

\section{Proof of Theorem\texorpdfstring{~\ref{teo:compact}}{3.2}}

\begin{lemma}\label{lem:limit}
Let $1<p<\infty$.  Assume that $b$ is a compactly supported Lipschitz function. Then 
\begin{equation}
    \lim_{m \to \infty}\|\mathcal{C}-\mathcal{C}_m\|_{L^p(dw) \longmapsto L^p(dw)}=0.
\end{equation}
\end{lemma}

\begin{proof}
Let $r_b>1$ be such that $\text{supp}\, b\subset B(0,r_b)$ {and let $L_b>0$ be such that $|b(\mathbf{x})-b(\mathbf{y})|\leq L_b\|\mathbf{x}-\mathbf{y}\|$ for all $\mathbf{x},\mathbf{y} \in \mathbb{R}^N$}. By  \eqref{eq:comm}, it is enough to prove that there is a constant $C>0$ such that for all positive integers $m$ such that $2^{m} \geq 2r_b$, $f \in L^p(dw)$, and $\mathbf{x} \in \mathbb{R}^N$, we have
\begin{equation}\label{eq:bxy_small}
    \sum_{\ell=-\infty}^{-m-1}\left|\int_{\mathbb{R}^N}(b(\mathbf{x})-b(\mathbf{y}))K_{\ell}(\mathbf{x},\mathbf{y})f(\mathbf{y})dw(\mathbf{y})\right| \leq C{L_b}2^{-\varepsilon m}\sum_{\sigma \in G}Mf(\sigma(\mathbf{x})),
\end{equation}
\begin{equation}\label{eq:bx}
    \sum_{\ell=m+1}^{\infty}\left|b(\mathbf{x})\int_{\mathbb{R}^N}K_{\ell}(\mathbf{x},\mathbf{y})f(\mathbf{y})dw(\mathbf{y})\right| \leq C{\|b\|_{L^{\infty}}}2^{-mN/p}\chi_{B(0,r_b)}(\mathbf{x})\|f\|_{L^p(dw)},
\end{equation}
\begin{equation}\label{eq:by}
    \sum_{\ell=m+1}^{\infty}\left|\int_{\mathbb{R}^N}K_{\ell}(\mathbf{x},\mathbf{y})b(\mathbf{y})f(\mathbf{y})dw(\mathbf{y})\right| \leq C{\|b\|_{L^{\infty}}}\left(\chi_{B(0,2^m)}(\mathbf{x})2^{-m\mathbf{N}}+\chi_{B(0,2^m)^c}(\mathbf{x})\|\mathbf{x}\|^{-\mathbf{N}}\right)\|f\|_{L^p(dw)}.
\end{equation}

Let us note that by Theorem~\ref{teo:support},  $\supp K_{\ell}(\mathbf{x},\cdot) \subseteq \mathcal{O}(B(\mathbf{x},2^{\ell}))$. Hence, by~\eqref{eq:scaled_CZ1} and the doubling property of the measure $dw$ (see~\eqref{eq:growth}) we have 
\begin{equation}\label{eq:K_l}
    |K_\ell(\mathbf{x},\mathbf{y}) |\leq C \left(1+ \frac{\|\mathbf x-\mathbf y\|}{2^{ \ell}}\right)^{-\varepsilon} w(B(\mathbf x,2^{\ell}))^{-1}.
\end{equation}

In order to prove~\eqref{eq:bxy_small}, we use now the Lipschitz condition for $b$ and get 
\begin{align*}
    &\sum_{\ell=-\infty}^{-m-1}\left|\int_{\mathbb{R}^N}(b(\mathbf{x})-b(\mathbf{y}))K_{\ell}(\mathbf{x},\mathbf{y})f(\mathbf{y})dw(\mathbf{y})\right| \\&\leq C\sum_{\ell=-\infty}^{-m-1}\frac{{L_b}}{w(B(\mathbf{x},2^{\ell}))}\int_{\mathcal{O}(B(\mathbf{x},2^{\ell}))}\|\mathbf{x}-\mathbf{y}\|^{\varepsilon}\frac{2^{\varepsilon \ell}}{\|\mathbf{x}-\mathbf{y}\|^{\varepsilon}}|f(\mathbf{y})|\,dw(\mathbf{y}) \leq C{L_b}2^{-\varepsilon m}\sum_{\sigma \in G}Mf(\sigma(\mathbf{x})).
\end{align*}
In order to prove~\eqref{eq:bx}, let us note that by~\eqref{eq:balls_asymp}, $w(B(\mathbf{x},2^{\ell})) \geq c2^{\ell N}$ for all $\ell\geq 0$. Hence, from~\eqref{eq:K_l}, H\"older's inequality, and the fact that $\supp b \subseteq B(0,r_b)$ we conclude that  
\begin{equation}\label{eq:Holder_calc}
\begin{split}
    & \sum_{\ell=m+1}^{\infty}\left|b(\mathbf{x})\int_{\mathbb{R}^N}K_{\ell}(\mathbf{x},\mathbf{y})f(\mathbf{y})\,dw(\mathbf{y})\right| \\
    &\leq C\chi_{B(0,r_b)}(\mathbf{x}) \|b\|_{L^\infty} \sum_{\ell=m+1}^{\infty}\left(\int_{\mathcal{O}(B(\mathbf{x},2^{\ell}))}\frac{1}{w(B(\mathbf{x},2^{\ell}))^{p'}}\,dw(\mathbf{y})\right)^{1/p'}\|f\|_{L^p(dw)} \\
    &\leq C\chi_{B(0,r_b)}(\mathbf{x}) \|b\|_{L^\infty} 2^{-mN/p}\|f\|_{L^p(dw)}.
\end{split}
\end{equation}
Finally, for~\eqref{eq:by}, we consider two cases. If $\|\mathbf{x}\| \leq 2^m$, then similarly to~\eqref{eq:Holder_calc} we get
\begin{equation}\label{eq:Holder_calc_1}
\begin{split}
    & \sum_{\ell=m+1}^{\infty}\left|\int_{\mathbb{R}^N}K_{\ell}(\mathbf{x},\mathbf{y})b(\mathbf{y})f(\mathbf{y})dw(\mathbf{y})\right| \\
    &\leq C\chi_{B(0,2^m)}(\mathbf{x} ) \|b\|_{L^\infty} \sum_{\ell=m+1}^{\infty}\left(\int_{B(0,r_b)}\frac{1}{w(B(\mathbf{y},2^{\ell}))^{p'}}\,dw(\mathbf{y})\right)^{1/p'}\|f\|_{L^p(dw)}\\
     &\leq C\chi_{B(0,2^m)}(\mathbf{x})\| b\|_{L^\infty} 2^{-m\mathbf{N}}\|f\|_{L^p(dw)}.
\end{split}
\end{equation}
Now we assume that $\|\mathbf{x}\|>2^m$. {Recall that $2^m\geq 2r_b$. From Theorem~\ref{teo:support} and the fact that $\supp K_{\ell} \subseteq B(0,2^{\ell})$, we conclude that 
$$ K_{\ell}(\mathbf{x},\mathbf{y})=0 \quad \text{ for all } \mathbf{y} \in B(0,r_b) \ \text{ and } \ell \geq m \ \text{ such that }  2^{\ell+1} <\|\mathbf{x}\|.$$ 
Further, by~\eqref{eq:K_l} and the doubling property~\eqref{eq:growth} of $dw$,
\begin{align*}
    |K_{\ell}(\mathbf{x},\mathbf{y})| \leq \frac{C}{w(B(0,2^\ell))} \quad \text{ for all } \mathbf{y} \in B(0,r_b) \ \text{ and } \ell \geq m \ \text{ such that }  2^{\ell+1} \geq \|\mathbf{x}\|.
\end{align*} }
Hence, by the fact that $b$ is supported by $B(0,r_b)$, we obtain
\begin{align*}
    &\sum_{\ell=m+1}^{\infty}\left|\int_{\mathbb{R}^N}K_{\ell}(\mathbf{x},\mathbf{y})b(\mathbf{y})f(\mathbf{y})dw(\mathbf{y})\right| \leq C\sum_{\ell=\lfloor\log(\|\mathbf{x}\|)\rfloor-2}^{\infty}\int_{B(0,r_b)}\frac{\| b\|_{L^\infty}}{w(B(0,2^{\ell}))}|f(\mathbf{y})|\,dw(\mathbf{y})\\&\leq C\frac{\| b\|_{L^\infty}}{w(B(0,\|\mathbf{x}\|))}\int_{B(0,r_b)}|f(\mathbf{y})|\,dw(\mathbf{y}) \leq C {r_b^{\mathbf N/p'}} \| b\|_{L^\infty} \|\mathbf{x}\|^{- \mathbf{N}}\|f\|_{L^p(dw)}.
\end{align*}
\end{proof}

\begin{proof}[Proof of Theorem\texorpdfstring{~\ref{teo:compact}}{ 3.2}]
Let us recall that the compactly supported Lipschitz functions form a dense subspace in $\text{VMO}$. Hence, by Theorem~\ref{teo:bounded}, it suffices to show that $\mathcal{C}$ is a compact operator on $L^p(dw)$ for any compactly supported  Lipschitz function $b$. 
Further, thanks to Lemma~\ref{lem:limit}, it is enough to prove that for any compactly supported Lipschitz function $b$, the commutators $\mathcal{C}_m$ are compact operators on $L^p(dw)$ for large enough positive integers $m$.  To this end, let $r_b>1$ be such that $\text{supp}\, b\subset B(0,r_b)$.  First we  note that if $f$ is supported by $B(0,r_b+2^m)^c$, then $b \cdot f \equiv 0$. Moreover, since $\supp \sum_{\ell=-m}^{m}K_{\ell}(\mathbf{x},\cdot) \subseteq \mathcal{O}(B(\mathbf{x},2^{m}))$ (cf. Theorem~\ref{teo:support}), we have $\sum_{\ell=-m}^{m}K_{\ell}(\mathbf{x},\cdot) \cdot f (\cdot ) \equiv 0$  for all $\mathbf{x} \in B(0,r_b)$. Consequently, for all $f \in L^p(dw)$ we have
\begin{align*}
    \mathcal{C}_mf(\mathbf{x})=\int_{\mathbb{R}^N}(b(\mathbf{x})-b(\mathbf{y}))\sum_{\ell=-m}^{m}K_{\ell}(\mathbf{x},\mathbf{y})f(\mathbf{y})\,dw(\mathbf{y})=\mathcal{C}_m(f \cdot \chi_{B(0,r_b+2^m)})(\mathbf{x}).
\end{align*}
Moreover, since $\supp \sum_{\ell=-m}^{m}K_{\ell}(\mathbf{x},\cdot) \subseteq \mathcal{O}(B(\mathbf{x},2^{m}))$,
\begin{align*}
    \supp \mathcal{C}_m(f)=\supp \mathcal{C}_m(f \cdot \chi_{B(0,r_b+2^m)}) \subseteq B(0,r_b+2^{m+1}).
\end{align*}
Let $\Omega=B(0,r_b+2^{m+1})$. The proof now is reduced to showing that $\mathcal{C}_m$ is a compact operator $L^p(\Omega,dw) \longmapsto L^p(\Omega,dw)$. The estimates~\eqref{eq:scaled_CZ1} and~\eqref{eq:scaled_CZ2} imply
\begin{align*}
    |\mathcal{C}_m f(\mathbf{x})|\leq C_{\Omega,m}\| f\|_{L^p(\Omega,dw)},
\end{align*}
\begin{align*}
     |\mathcal{C}_mf(\mathbf{x})-\mathcal{C}_m f(\mathbf{x}')|\leq C_{\Omega,m} \|\mathbf{x}-\mathbf{x}'\|^{\varepsilon} \cdot\| f\|_{L^p(\Omega,dw)}
\end{align*}
for all $\mathbf{x},\mathbf{x}' \in \mathbb{R}^N$. By the Arzel\'a-Ascoli  theorem the set
\begin{align*}
\{\mathcal{C}_m f: \ f\in L^p(\Omega,dw), \| f\|_{L^p(\Omega,dw)}\leq 1\}    
\end{align*}
is relatively compact in $C(\Omega)$ with the $\sup$-norm. Hence it is relatively compact in $L^p(\Omega,dw)$, since $w(\Omega)$ is finite. 
\end{proof}

\end{document}